\documentclass[12pt]{amsart}

\usepackage[left=1.4in,right=1.4in,top=0.75in,bottom=0.75in]{geometry}
\usepackage[T1]{fontenc}
\usepackage{fourier}

\usepackage{graphicx}

\usepackage{amsthm}
\usepackage{amsmath}
\usepackage{amssymb}
\usepackage{pinlabel}

\usepackage{microtype}

\usepackage{caption}
\usepackage{wrapfig}

\usepackage[hidelinks,pagebackref]{hyperref}

\renewcommand*{\backref}[1]{}
\renewcommand*{\backrefalt}[4]{
  \ifcase #1
  [No citations.]
  \or [#2]
  \else [#2]
  \fi }

\makeatletter
\let\c@equation\c@subsection
\makeatother
\numberwithin{equation}{section}

\makeatletter
\let\c@figure\c@equation
\makeatother
\numberwithin{figure}{section}

\theoremstyle{plain}
\newtheorem{theorem}[equation]{Theorem}

\newtheorem{lemma}[equation]{Lemma}

\newtheorem{proposition}[equation]{Proposition}

\newtheorem{criterion}[equation]{Criterion}

\theoremstyle{definition}
\newtheorem{definition}[equation]{Definition}
\newtheorem{remark}[equation]{Remark}

\newtheorem*{claim*}{Claim}

\newcommand{\refthm}[1]{Theorem~\ref{Thm:#1}}
\newcommand{\reflem}[1]{Lemma~\ref{Lem:#1}}
\newcommand{\refprop}[1]{Proposition~\ref{Prop:#1}}

\newcommand{\reffig}[1]{Figure~\ref{Fig:#1}}
\newcommand{\refcri}[1]{Criterion~\ref{Cri:#1}}
\newcommand{\refdef}[1]{Definition~\ref{Def:#1}}
\newcommand{\refsec}[1]{Section~\ref{Sec:#1}}

\newcommand{\fakeenv}{} %%% prints the emptystring
\newenvironment{restate}[2]
{
 \renewcommand{\fakeenv}{#2} %%% So now \fakeenv prints #2
 \theoremstyle{plain}
 \newtheorem*{\fakeenv}{#1~\ref{#2}}
 \begin{\fakeenv}
}
{
 \end{\fakeenv}
}

\newcommand{\st}{\mathbin{\mid}} %%% \mathbin = binary operator
\newcommand{\from}{\colon} % As in ``f maps _from_ X _to_ Y''.
\newcommand{\cross}{\times}

\newcommand{\RR}{\mathbb{R}}

\newcommand{\calB}{\mathcal{B}}
\newcommand{\calC}{\mathcal{C}}

\newcommand{\bdy}{\partial}
\newcommand{\carr}{\prec}

\newcommand{\eff}{\dashv}

\newcommand{\supp}{\operatorname{supp}}
\newcommand{\op}{{\operatorname{op}}}
\newcommand{\ML}{\operatorname{ML}}

\newcommand{\ind}{\operatorname{index}}
\newcommand{\interior}{\operatorname{interior}}

\newcommand{\Hyp}{{\sf \delta}}
\newcommand{\DiamV}{{\sf K_0}}
\newcommand{\Quasi}{{\sf K_1}}

\title{The curves not carried}

\author[Gadre]{Vaibhav Gadre}
\address{\hskip-\parindent
    	Department of Mathematics\\
        University of Warwick\\
        Coventry, UK}
\email{v.gadre@warwick.ac.uk}

\author[Schleimer]{Saul Schleimer}
\address{\hskip-\parindent
    	Department of Mathematics\\
        University of Warwick\\
        Coventry, UK}
\email{s.schleimer@warwick.ac.uk}

\thanks{This work is in the public domain.}  

\thanks{The first author was supported by a Global Research Fellowship
  from the Institute of Advanced Study at the University of Warwick.}

\begin{document}

%%% As simple as possible --
%%%   short sentences, minimize use of math mode, no \cites.
\begin{abstract}
Suppose $\tau$ is a train track on a surface $S$.  Let $\calC(\tau)$
be the set of isotopy classes of simple closed curves carried by
$\tau$.  Masur and Minsky [2004] prove that $\calC(\tau)$ is
quasi-convex inside the curve complex $\calC(S)$.  We prove that the
complement, $\calC(S) - \calC(\tau)$, is quasi-convex.
\end{abstract}

%%% AMS MSC - 57M99 low-dim topology, 30F60 - Teichmuller theory,
%%% 20F65 - Geometric group theory

%%% Key words - train tracks, curve complex, quasi-convex

\maketitle

\section{Introduction}

The curve complex $\calC(S)$, of a surface $S$, is deeply important in
low-dimensional topology.  One foundational result, due to Masur and
Minsky, states that $\calC(S)$ is Gromov
hyperbolic~\cite[Theorem~1.1]{MasurMinsky99}.

Suppose $\tau$ is a train track on $S$.  The set $\calC(\tau) \subset
\calC(S)$ consists of all curves $\alpha$ carried by $\tau$: we write
this as $\alpha \carr \tau$.  Another striking result of Masur and
Minsky is that $\calC(\tau)$ is quasi-convex in $\calC(S)$.  This
follows from hyperbolicity and their result that splitting sequences
of train tracks give rise to quasi-convex subsets in
$\calC(S)$~\cite[Theorem~1.3]{MasurMinsky04}.

We prove a complementary result.

\begin{restate}{Theorem}{Thm:QuasiConvexGS}
Suppose $\tau \subset S$ is a train track.  Then the curves not
carried by $\tau$ form a quasi-convex subset of $\calC(S)$.
\end{restate}

This supports the intuition that, for a maximal birecurrent track
$\tau$, the carried set $\calC(\tau)$ is like a half-space in a
hyperbolic space.  
%%% So, it is unlike a horoball.

When $S$ is the four-holed sphere or once-holed torus the proof is an
exercise in understanding how $\calC(\tau)$ sits inside the Farey
graph.
%%% Solution for $S_{1,1}$: any complete track in S = S_{1,1} has two
%%% switches and three branches, one large and two small.  There are
%%% exactly two carried curves \alpha and \alpha' that are not fully
%%% carried, and these are the vertex cycles.  Note that i(\alpha,
%%% \alpha') = 1 so they are adjacent in the Farey graph.  The
%%% laminations carried by $\tau$ form an interval in PML with
%%% endpoints \alpha and \alpha'.  The laminations not carried is the
%%% complement.  The curves not carried lie in this interval - they
%%% form a half-space in the Farey graph.  This proves
%%% \refthm{QuasiConvexGS}.  Note that there is exactly one curve
%%% \beta, not carried by \tau, meeting \alpha and \alpha' each in a
%%% single point.  This proves \refprop{Face}.  Note that a non-empty,
%%% non-complete track \tau in S is necessarily a simple loop.  This
%%% can be folded (in infinitely many different ways!) to obtain a
%%% complete track.  This gives \refprop{Fold}.  
%%% 
%%% Solution for S = S_{0,4}: Let v, w be vertex cycles of \tau (there
%%% are at least two!).  If supp(v) is a simple closed curve then S -
%%% supp(v) = P_0 \cap P_1, both pants.  So w gives a wave in each of
%%% these and we have found all of \tau.  If supp(v) is a barbell then
%%% S - supp(s) is a pair of monogons and a pants, so w is even
%%% simpler.  In each case we have to show that there are no more
%%% vertices.  This done, we have pv + qw gives all of the carried
%%% curves, and we are done.  There is probably an argument using the
%%% fact that S_{1,1} is a double branched cover of S_{0,4}.
In what follows we suppose that $S$ is a connected, compact, oriented
surface with $\chi(S) \leq -2$, and not a four-holed sphere.  Here is
a rough sketch of the proof of \refthm{QuasiConvexGS}.  Suppose
$\gamma$ and $\gamma'$ are simple closed curves, not carried by
$\tau$.  Let $[\gamma, \gamma']$ be a geodesic in $\calC(S)$.  Suppose
$\alpha$ and $\alpha'$ are the first and last curves of $[\gamma,
  \gamma']$ carried by $\tau$.  Fix splitting sequences from $\tau$ to
$\alpha$ and $\alpha'$, respectively.  For each splitting sequence,
the vertex sets form a $\Quasi$--quasi-convex subset inside
$\calC(S)$.  Since $\calC(S)$ is Gromov hyperbolic, the geodesic
segment $[\alpha, \alpha']$ is $\Quasi + \Hyp$--close to the union of
vertex sets.  \refprop{Closing} completes the proof by showing each
vertex cycle, along each splitting sequence, is uniformly close to a
non-carried curve.

Before stating \refprop{Closing} we recall a few definitions.  A train
track $\tau \subset S$ is \emph{large} if all components of $S - \tau$
are disks or peripheral annuli.  A track $\tau$ is \emph{maximal} if
it is not a proper subtrack of any other track.  The \emph{support},
$\supp(\alpha, \tau)$, of a carried curve $\alpha \carr \tau$ is the
union of the branches of $\tau$ along which $\alpha$ runs.

\begin{restate}{Proposition}{Prop:Closing}
Suppose $\tau \subset S$ is a train track and $\alpha \carr \tau$ is a
carried curve.  Suppose $\supp(\alpha, \tau)$ is large, but not
maximal.  Then there is an essential, non-peripheral curve $\beta$ so
that $i(\alpha, \beta) \leq 1$ and any curve isotopic to $\beta$ is
not carried by $\tau$.
\end{restate}

The idea behind \refprop{Closing} is as follows.  Since $\sigma =
\supp(\alpha, \tau)$ is large all components of $S - \sigma$ are disks
or peripheral annuli.  Since $\sigma$ is not maximal there is a
component $Q \subset S - \sigma$ which is not an ideal triangle or a
once-holed ideal monogon.  Hence, there is a \emph{diagonal} $\delta$
of $Q$ that is not carried by $\tau$.  We then extend $\delta$, in a
purely local fashion, to a simple closed curve $\beta$.  By
construction $\beta$ is in \emph{efficient position} with respect to
$\tau$ and meets $\alpha$ at most once.  Finally, we appeal to
Criteria~\ref{Cri:AcrossArc} or~\ref{Cri:AcrossCorner} to show that
$\beta$ is not isotopic to a carried curve.

\section{Background}

We review the basic definitions needed for the rest of the paper.
Throughout we suppose $S$ is a compact, connected, smooth, oriented
surface.

%%% Equip the unit circle $S^1 \subset \CC$ with its canonical
%%% counterclockwise orientation.  A smooth embedding $\alpha \from
%%% S^1 \to S$ is called a \emph{simple closed curve}.  The
%%% equivalence class $[\alpha]$ is generated by reversing orientation
%%% and by smooth isotopy.  We say $\alpha$ is \emph{inessential} if
%%% its image cuts a disk off of $S$.  We say $\alpha$ is
%%% \emph{peripheral} if its image cuts an annulus off of $S$.

\subsection{Corners and index}

Suppose $R \subset S$ is a subsurface with piecewise smooth boundary.
The non-smooth points of $\bdy R$ are the \emph{corners} of $R$.  We
require that the exterior angle at each corner be either $\pi/2$ or
$3\pi/2$, giving \emph{inward} and \emph{outward} corners.  Let
$c_\pm(R)$ count the inward and outward corners of $R$, respectively.
The \emph{index} of $R$ is
\[
\ind(R)= \chi(R) + \frac{c_+(R)}{4} - \frac{c_-(R)}{4}.
\]
For example, if $R$ is a rectangle then its index is zero.  In
general, if $\alpha \subset R$ is a properly embedded, separating arc,
avoiding the corners of $R$, and orthogonal to $\bdy R$, and if $P$
and $Q$ are the closures of the components of $R - \alpha$, then we
have $\ind(P) + \ind(Q) = \ind(R)$.

%%% If $R$ is a region with corners, $\bdy R$ may not have a sensible
%%% division into vertical and horizontal.

\subsection{The curve complex}

Define $i(\alpha, \beta)$ to be the geometric intersection number
between a pair of simple closed curves.
%%% The minimal intersection number between isotopy representatives of
%%% the classes $[\alpha]$ and $[\beta]$.  Henceforth we will only
%%% distinguish between a simple closed curve and its image when
%%% necessary.
The \emph{complex of curves} $\calC(S)$ is, for us, the following
graph.  Vertices are essential, non-peripheral isotopy classes of
simple closed curves.  Edges are pairs of distinct vertices $\alpha$
and $\beta$ where $i(\alpha, \beta) = 0$.  When $\chi(S) \leq -2$ (and
$S$ is not the four-holed sphere) it is an exercise to show that
$\calC(S)$ is connected.  We may equip $\calC(S)$ with the usual edge
metric, denoted $d_S$.
%%% $d_S(\alpha, \beta)$ is the number of edges in a minimal edge path
%%% between $\alpha$ and $\beta$.
Here is a foundational result due to Masur and Minsky.

\begin{theorem}
\label{Thm:Hyperbolic}
\cite[Theorem~1.1]{MasurMinsky99}
The curve complex $\calC(S)$ is Gromov hyperbolic.
\end{theorem}

\subsection{Train tracks}

%%% A standard reference is the book by Penner, with
%%% Harer~\cite{PennerHarer92}.  We don't use the book at all now.
%%% Perhaps add a reference to Lee's monograph instead - DONE. 
A \emph{pre-track} $\tau \subset S$ is a non-empty finite embedded
graph with various properties as follows.
%%% If S is non-compact we replace finiteness with local finiteness,
%%% or with concrete local models.  Anyway - leave finiteness in. 
The vertices (called \emph{switches}) are all of valence three.  The
edges (called \emph{branches}) are smoothly embedded.  Any point $x$
lying in the interior of a branch $A \subset \tau$ divides $A$ into a
pair of \emph{half-branches}.  At a switch $s \in \tau$, we may orient
the three incident half-branches $A$, $B$, and $C$ away from $s$.
After renaming the branches, if necessary, their tangents satisfy
$V(s, A) = -V(s, B) = -V(s, C)$.  We say $A$ is a \emph{large}
half-branch and $B$ and $C$ are \emph{small}.  This finishes the
definition of a pre-track.  See Figures~\ref{Fig:Split}
and~\ref{Fig:Shift} for various local pictures of a pre-track.

A branch $B \subset \tau$ is either \emph{small}, \emph{mixed}, or
\emph{large} as it contains zero, one, or two large half-branches.  We
may \emph{split} a pre-track $\tau$ along a large branch, as shown in
\reffig{Split}, to obtain a new track $\tau'$.  Conversely, we
\emph{fold} $\tau'$ to obtain $\tau$.  If a branch is mixed then we
may \emph{shift} along it to obtain $\tau'$, as shown in
\reffig{Shift}.  Note shifting is symmetric; if $\tau'$ is a shift of
$\tau$ then $\tau$ is a shift of $\tau'$.
%%% Note that we do allow shifting, below. 

\begin{figure}[htbp]
\includegraphics[width=0.7\textwidth]{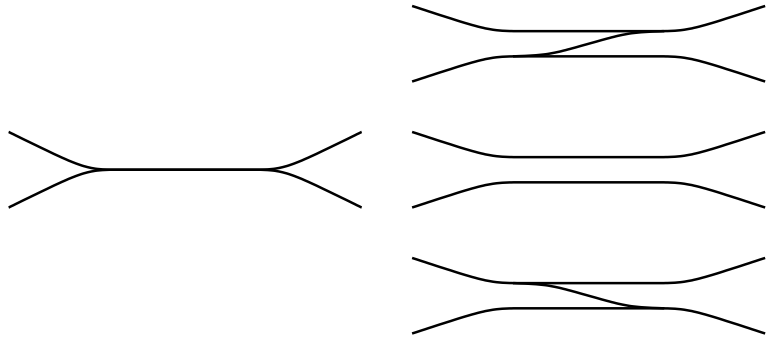}
\caption{A large branch admits a left, central, or right splitting.}
%%% Right, left splits, and shifts are planar rotations of the tree
%%% with four leaves.  Central splittings are also called collisions.
\label{Fig:Split}
\end{figure}

\begin{figure}[htbp]
\includegraphics[width=0.7\textwidth]{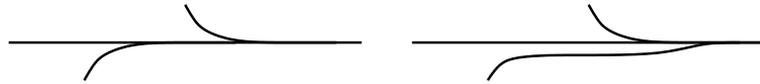}
\caption{A mixed branch admits a shift.}
\label{Fig:Shift}
\end{figure}

Suppose $\tau \subset S$ is a pre-track.  We define $N = N(\tau)$, a
\emph{tie neighborhood} of $\tau$ as follows.  For every branch $B$ we
have a rectangle $R = R_B = B \cross I$.  For all $x \in B$ we call
$\{x\} \cross I$ a \emph{tie}.  The two ties of $\bdy B \cross I$ are
the \emph{vertical boundary} $\bdy_v R$ of $R$.  The boundaries of all
of the ties form the \emph{horizontal boundary} $\bdy_h R$ of $R$.
Any tie $J \subset R$, meeting the interior of $R$, cuts $R$ into a
pair of \emph{half-rectangles}.  The points $\bdy \bdy_v R = \bdy
\bdy_h R$ are the corners of $R$; all four are outward corners.

%%% Would be nicer to have the caption on the side...

\begin{wrapfigure}[14]{r}{0.45\textwidth}
\vspace{-10pt}
\centering 
\includegraphics[width = 0.35\textwidth, angle=180]{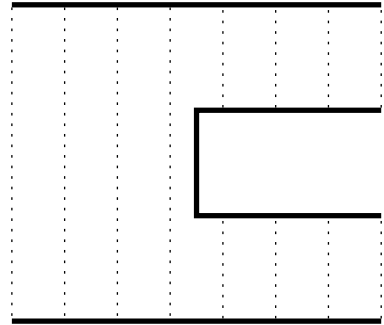}
\caption{The local model for $N(\tau)$ near a switch.  The dotted
  lines are ties.}
\label{Fig:Tie}
\end{wrapfigure}

We embed all of the rectangles $R_B$ into $S$ as follows.  Suppose $A$
(large) and $B$ and $C$ (small) are the half-branches incident to the
switch $s$.  The vertical boundary of $R_B$ (respectively $R_C$) is
glued to the upper (lower) third of the vertical boundary of $R_A$.
See \reffig{Tie}.  The resulting tie neighborhood $N = N(\tau)$ has
horizontal boundary $\bdy_h N = \bigcup \bdy_h R_B$.  The vertical
boundary of $N$ is the closure of $\bdy N - \bdy_h N$.  Again $\bdy
\bdy_v N$ is the set of corners of $N$; all of these are inward
corners.  
%%% Note that all ties are orthogonal to $\bdy_h N$ at their
%%% endpoints.  
We use $n(\tau)$ to denote the interior of $N(\tau)$.  We may now give
our definition of a train track.

\begin{definition}
Suppose $\tau \subset S$ is a pre-track and $N(\tau)$ is a tie
neighborhood.  We say $\tau$ is a \emph{train track} if
every component of $S - n(\tau)$ has negative index. 
%%% no smooth frontier - every component of $\bdy N(\tau)$ contains a
%%% corner.  [If we required that, then a simple closed curve is not a
%%% track. Also, it is difficult to say where this is used.  Note that
%%% large tracks automatically have this property.  Hmm.  Perhaps in
%%% the future, we'll need to rule out smooth boundary when extending
%%% tracks.  That is, we'd really like diagonals to start and end on
%%% vertical boundary components.]
\end{definition}

A track $\tau \subset S$ is \emph{large} if every component of $S -
n(\tau)$ is either a disk or a peripheral annulus.  A track $\tau$ is
\emph{maximal} if every component of $S - n(\tau)$ is either a hexagon
or a once-holed bigon.
%%% Ie an ``ideal triangle'' or a ``once-holed ideal monogon''.

\subsection{Carried curves and transverse measures}
Suppose $\alpha \subset S$ is a simple closed curve.  If $\alpha
\subset N(\tau)$ and $\alpha$ is transverse to the ties of $N(\tau)$
then we say $\alpha$ is \emph{carried} by $\tau$.  We write this as
$\alpha \carr \tau$.  
%%% An isotopy class $[\beta]$ of curves is \emph{not carried} by
%%% $\tau$ if no representative is carried.  
It is an exercise to show that if $\alpha$ is carried then $\alpha$ is
essential and non-peripheral.
%%% Use the fact that index is additive and that the index of N(\tau)
%%% is zero.
We define $\calC(\tau) = \{ \alpha \in \calC(S) \st \alpha \carr
\tau\}$.  Note that $\calC(\tau)$ is non-empty.

Let $\calB = \calB_\tau$ be the set of branches of $\tau$.  Fix a
switch $s$ and suppose that the half-branches $A$, $B$, and $C$ are
adjacent to $s$, with $A$ being large.  A function $\mu \from \calB
\to \RR_{\geq 0}$ satisfies the \emph{switch equality} at $s$ if
\[
\mu(A) = \mu(B) + \mu(C).
\]
We call $\mu$ a \emph{transverse measure} if $\mu$ satisfies all
switch equalities.  For example, any carried curve $\alpha \carr \tau$
gives an integral transverse measure $\mu_\alpha$.
%%% Count intersections with any tie in $R_B$ to get $\mu_\alpha(B)$.
This permits us to define $\sigma = \supp(\alpha, \tau)$, the
\emph{support} of $\alpha$ in $\tau$: a branch $B \subset \tau$ lies
in $\sigma$ if $\mu_\alpha(B) > 0$.
%%% Note that the support may have a smooth horizontal boundary.

Here is a ``basic observation'' from~\cite[page~117]{MasurMinsky99}.

\begin{lemma}
\label{Lem:Basic}
Suppose $\tau$ is a maximal train track and suppose $\alpha \carr
\tau$ has full support: $\tau = \supp(\alpha, \tau)$.  Suppose $\beta$
is an essential, non-peripheral curve with $i(\alpha, \beta) = 0$.
Then $\beta$ is also carried by $\tau$.  \qed
\end{lemma}

Since the switch equalities are homogeneous the set of solutions
$\ML(\tau)$ is a rational cone.  We projectivize $\ML(\tau)$ to obtain
$P(\tau)$, a non-empty convex polytope.  All vertices of $P(\tau)$
arise from carried curves; we call such curves \emph{vertex cycles}
for $\tau$.  Thus the set $V(\tau)$ of vertex cycles is naturally a
subset of $\calC(\tau) \subset \calC(S)$.  Deduce if $\tau'$ is a
shift of $\tau$ then $V(\tau') = V(\tau)$.

\begin{figure}[htbp]
\includegraphics[width=0.7\textwidth]{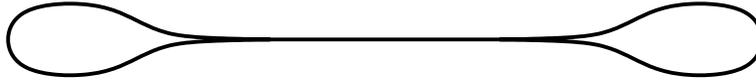}
\caption{A \emph{barbell}: a train track with one large branch and two
  small branches, where the midpoint of the large branch separates.}
\label{Fig:Barbell}
\end{figure}

\begin{lemma}
\label{Lem:Vertex}
A carried curve $\alpha \carr \tau$ is a vertex cycle if and only if
$\supp(\alpha, \tau)$ is either a simple close curve or a barbell (see
\reffig{Barbell}).
\end{lemma}

\begin{proof}
The forward direction is given by Proposition~3.11.3(3)
of~\cite{Mosher03}. 
%%% Here is an argument relying on surgery.  In every branch we see at
%%% most two arcs of $\alpha$, and these are oriented oppositely.  In
%%% fact, we can assume that the orientations of the two arcs are
%%% always anti-clockwise, using a futher surgery arguement.  Pick one
%%% tie per large branch with weight two.  This gives a chord diagram.
%%% If any pair links, then there is a surgery.  Finally, if there are
%%% at least two chords, then there are at least two outmost chords.
%%% So double and surger.//
The backward direction is an exercise in the definitions.
%%% If \alpha \carr \tau then any decomposition of \alpha inside of
%%% \tau (that is, as a sum of vertex cycles) is in fact a
%%% decomposition inside of \supp(\alpha, \tau).  If the support is a
%%% curve or barbell, then the support has only one vertex cycle -
%%% that is, \alpha.  
\end{proof}

The usual upper bound on distance in $\calC(S)$, coming from geometric
intersection number~\cite[Lemma~1.21]{Schleimer06b}, gives the
following.

\begin{lemma}
\label{Lem:Diameter}
For any surface $S$ there is a constant $\DiamV$ with the following
property.  Suppose $\tau$ is a track.  Suppose $\sigma$ is a split,
shift, or subtrack of $\tau$.  Then the diameter of $V(\tau) \cup
V(\tau')$ inside of $\calC(S)$ is at most $\DiamV$.  \qed
\end{lemma}

%%% Note that Lemmas~\ref{Lem:Vertex} and~\ref{Lem:Diameter} also hold if
%%% $\tau$ is a pre-track which is in turn a subtrack of a train track.

\subsection{Quasi-convexity}

A subset $A \subset \calC(S)$ is $K$--\emph{quasi-convex} if for every
$\alpha$ and $\beta$ in $A$, any geodesic $[\alpha, \beta] \subset
\calC(S)$ lies within a $K$--neighborhood of $A$.  Recall if $A$ and
$B$ are $K$--quasi-convex sets in $\calC(S)$, and if $A \cap B$ is
non-empty, then the union $A \cup B$ is $K + \Hyp$--quasi-convex.  
%%%
%%% A \emph{train-track sequence} $\{\tau_i\}$ has the property that
%%% $\tau_{i+1}$ is a split, a shift, or a subtrack of $\tau_i$.
%%%
We now have a more difficult result.

\begin{theorem}
\label{Thm:QuasiConvexMM}
\cite[Theorem~1.3]{MasurMinsky04} For any surface $S$ there is a
constant $\Quasi$ with the following property.  Suppose that
$\{\tau_i\}$ is sequence where $\tau_{i+1}$ is a split, shift, or
subtrack $\tau_i$.  Then the set $V = \bigcup_i V(\tau_i)$ is
$\Quasi$--quasi-convex in $\calC(S)$. \qed
\end{theorem}

\begin{remark}
In the first statement of their
Theorem~1.3~\cite[page~310]{MasurMinsky04} Masur and Minsky assume
their tracks are large and recurrent.  However, as they remark after
their Lemma~3.1, largeness is not necessary.  Also, it is an exercise
to eliminate the hypothesis of recurrence, say by using
\reflem{Diameter} and the subtracks $\supp(\alpha, \tau_i)$ (for any
fixed curve $\alpha \in \bigcap P(\tau_i)$).

A more subtle point is that their Lemmas~3.2, 3.3, and~3.4 use the
train-track machinery of another of their papers~\cite{MasurMinsky99}.
Transverse recurrence is used in an essential way in the second
paragraph of the proof of Lemma~4.5 of that earlier paper.  However
the crucial ``nesting lemma''~\cite[Lemma~3.4]{MasurMinsky04} can be
proved without transverse recurrence.  This is done in Lemma~3.2
of~\cite{GadreTsai11}.

Thus, as stated above, \refthm{QuasiConvexMM} does not require any
hypothesis of largeness, recurrence, or transverse recurrence.
\end{remark}

\section{Proof of the main theorem}

We now have enough tools in place to see how \refprop{Closing} implies
our main result.

\begin{theorem}
\label{Thm:QuasiConvexGS}
Suppose $\tau \subset S$ is a train track.  The curves not carried by
$\tau$ form a quasi-convex subset of $\calC(S)$.
\end{theorem}

%%% Suppose $\sigma$ is not large. Then, $\text{diam}(\calC(\sigma))
%%% \leq 2$.

\begin{proof}
We may assume $\chi(S) \leq -2$, and that $S$ is not a four-holed
sphere.
%%% Thus adjacency in \calC is disjointness (and not isotopic).
Suppose $\gamma, \gamma' \in \calC(S)$ are not carried by $\tau$.  Fix
a geodesic $[\gamma, \gamma']$ in $\calC(S)$.  If $[\gamma, \gamma']$
is disjoint from $\calC(\tau)$ there is nothing to prove.

So, instead, suppose $\alpha$ and $\alpha'$ are the first and last
curves, along $[\gamma, \gamma']$, carried by $\tau$.  Let $\beta$ be
the predecessor of $\alpha$ in $[\gamma, \gamma']$ and let $\beta'$ be
the successor of $\alpha'$.  Thus, $\beta$ and $\beta'$ are not
carried by $\tau$.  The contrapositive of \reflem{Basic} now implies
that the tracks $\supp(\alpha, \tau)$ and $\supp(\alpha', \tau)$ are
not maximal.
%%% This is where the assumption S \neq S_{1,1}, S \neq S_{0,4} is
%%% used.

%%% Changes in this paragraph. 
For the moment, we fix our attention on $\alpha$.  We choose a
splitting and shifting sequence $\{\tau_i\}_{i = 0}^n$ with the
following properties:
\begin{itemize}
\item
$\tau_0 = \tau$,
\item
for all $i$, the curve $\alpha$ is carried by $\tau_i$, and 
\item
$\supp(\alpha, \tau_n)$ is a simple closed curve.
\end{itemize}
We find a similar sequence $\{\tau_i'\}$ for $\alpha'$.

Let $V = \bigcup V(\tau_i)$ be the vertices of the splitting sequence
$\{\tau_i\}$; define $V'$ similarly.  The hyperbolicity of $\calC(S)$
(\refthm{Hyperbolic}) and the quasi-convexity of vertex sets
(\refthm{QuasiConvexMM}) imply the geodesic $[\alpha, \alpha']$ lies
within a $\Quasi + \Hyp$--neighborhood of $V \cup V'$.  To finish the
proof we must show that every vertex of $V$ (and of $V'$) is close to
a non-carried curve of $\tau$.

Using \reflem{Vertex} twice we may pick vertex cycles $\alpha_i \in
V(\tau_i)$ so that:
\begin{itemize}
\item
$\alpha_n = \alpha$ and
\item
$\alpha_i \carr \supp(\alpha_{i+1}, \tau_i)$.
\end{itemize}
%%% Claim - there exists such $\alpha_i$.  Pf: If $\supp(\alpha_{i+1},
%%% \tau_i)$ is a loop or a barbell, then $\alpha_i = \alpha_{i+1}$.
%%% If the support is larger, then it contains a loop or barbell and
%%% this gives $\alpha_i$.
Define $\sigma_i = \supp(\alpha_i, \tau)$.  By construction
$\supp(\alpha_i, \tau_i) \subset \supp(\alpha_{i+1}, \tau_i)$.  If we
fold backwards along the sequence then, the former track yields
$\sigma_i$ while the latter yields $\sigma_{i+1}$.  We deduce
$\sigma_i \subset \sigma_{i+1}$.  Recall that $\sigma_n =
\supp(\alpha, \tau)$ is not maximal.  Thus none of the $\sigma_i$ are
maximal.

Let $m = \max \big\{ \ell \st \mbox{$\sigma_\ell$ is small} \big\}$.
Fix any curve $\omega \in \calC(S)$ disjoint from $\sigma_m$.  Using
$\omega$ we deduce $d_S(\alpha_i, \alpha_m) \leq 2$, for any $i \leq
m$.

If $m = n$ then \reflem{Diameter} implies the set $V = \bigcup V(\tau_i)$
lies within a $\DiamV + 3$--neighbor\-hood of $\beta$, and we are done.

So we may assume that $m < n$.  In this case \reflem{Diameter} implies
the set $\bigcup_{i=0}^m V(\tau_i)$ lies within a $2\DiamV+
2$--neighborhood of $\alpha_{m+1}$.  Recall $\alpha_i \carr \tau$ and
$\sigma_i = \supp(\alpha_i, \tau)$ is assumed to be a large, yet not
maximal, subtrack of $\tau$.  Thus we may apply \refprop{Closing} to
obtain a curve $\beta_i$ so that:
\begin{itemize}
\item
$\beta_i \in \calC(S) - \calC(\tau)$ and 
\item
$i(\alpha_i, \beta_i) \leq 1$.
\end{itemize}
Applying \reflem{Diameter} we deduce, whenever $i > m$, that
$V(\tau_i)$ lies within a $\DiamV + 2$--neighborhood of $\beta_i$.

The same argument applies to the splitting sequence from $\tau$ to
$\alpha'$.  This completes the proof of the theorem.
\end{proof}

\section{Efficient position}

In order to prove \refprop{Closing}, we here give criteria to show
that a curve $\beta$ \emph{cannot} be carried by a given track $\tau$.
We state these in terms of \emph{efficient position}, defined
previously in~\cite[Definition 2.3]{MasurEtAl12}.  See
also~\cite[Definition~3.2]{Takarajima00a}.

Suppose $\tau$ is a train track and $N = N(\tau)$ is a tie
neighborhood.  A simple arc $\gamma$, properly embedded in $N$, is a
\emph{carried arc} if it is transverse to the ties and disjoint from
$\bdy_h N$.  

\begin{definition}
\label{Def:EffPos}
Suppose $\beta \subset S$ is a properly embedded arc or curve which is
transverse to $\bdy N$ and disjoint from $\bdy \bdy_v N$, the corners
of $N$.  Then $\beta$ is in \emph{efficient position} with respect to
$\tau$, written $\beta \eff \tau$, if
\begin{itemize}
\item
every component of $\beta \cap N(\tau)$ is carried or is a tie and
\item 
every component of $S - n(\beta \cup \tau)$ has negative index or is
a rectangle.
%%% No annuli allowed.
\end{itemize}
\end{definition}

Here $n(\beta \cup \tau)$ is a shorthand for $n(\beta) \cup n(\tau)$,
where the ties of $n(\beta)$ are either subties of, or orthogonal to,
ties of $n(\tau)$.
%%% Note that carried curves are automatically in efficient position.
An index argument proves if $\beta \eff \tau$ then $\beta$ is
essential and non-peripheral.  See~\cite[Lemma~2.5]{MasurEtAl12}.

\begin{criterion}
\label{Cri:AcrossArc}
Suppose $\beta \eff \tau$ is a curve.  Orient $\beta$.  Suppose there
are regions $L$ and $R$ of $S - n(\beta \cup \tau)$ and a component
$\beta_M \subset \beta - n(\tau)$ with the following properties.
\begin{itemize}
\item
$L$ and $R$ lie immediately to the left and right, respectively, of
  $\beta_M$ and
\item
$L$ and $R$ have negative index. 
\end{itemize}
Then any curve isotopic to $\beta$ is not carried by $\tau$.
\end{criterion}
%%% Note that we allow L = R.

\begin{figure}[htbp]
\labellist
\small\hair 2pt
\pinlabel {$L$} at 107 107
\pinlabel {$R$} at 163 163
\pinlabel {$\beta_M$} [tr] at 166 110
\endlabellist
\includegraphics[height = 7 cm]{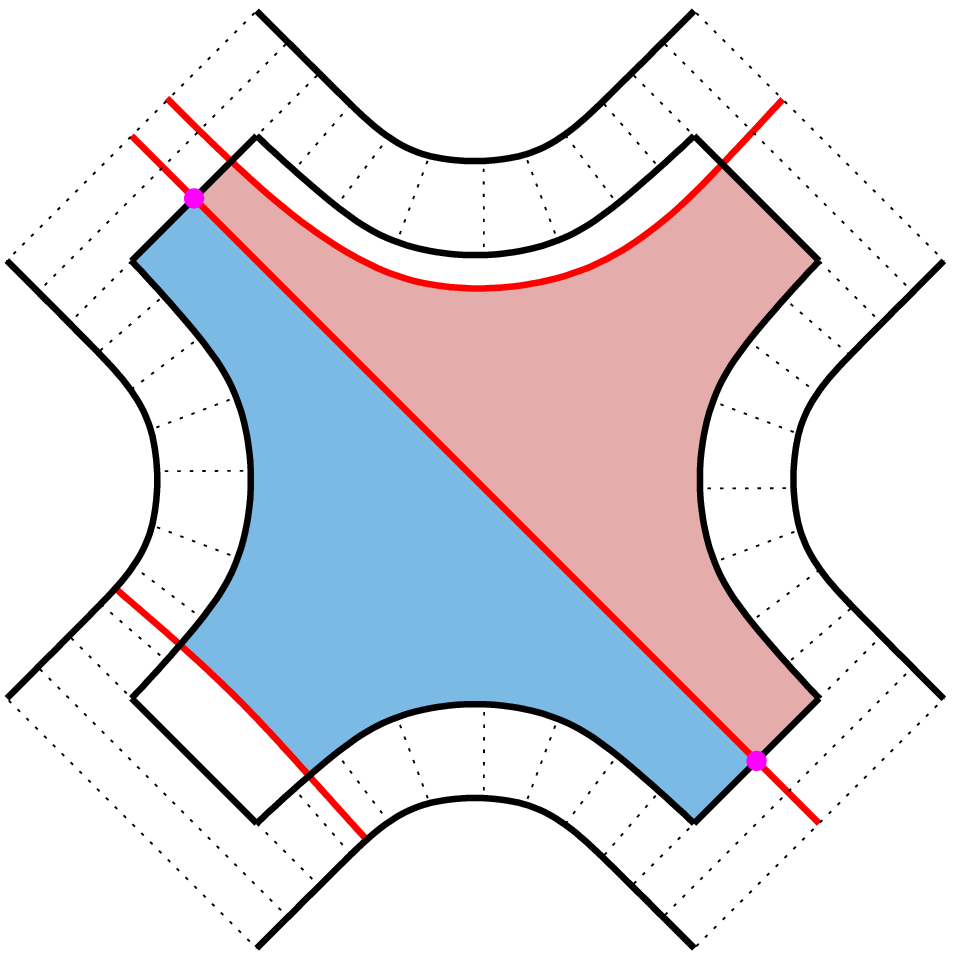}
\qquad
\begin{minipage}{0.33\textwidth}
\vspace{-6.5cm}
\labellist
\small\hair 2pt
%\pinlabel {$N(\tau)$} [bl] at 128 56
\pinlabel {$L$} at 32 119
\pinlabel {$R$} at 94 28
\pinlabel {$\beta_I$} [l] at 65 73
\endlabellist
\includegraphics[height = 5cm]{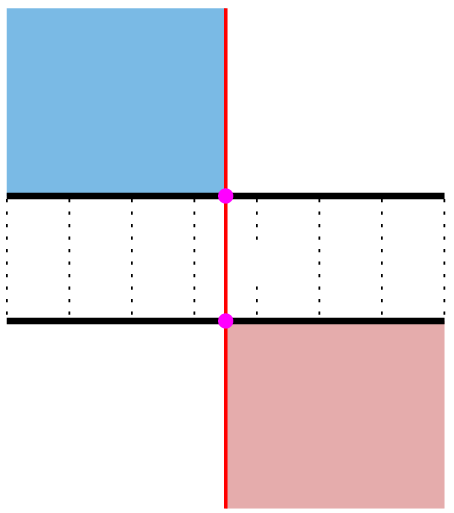}
\end{minipage}
\caption{Left: The regions $L$ and $R$ are both adjacent to the arc
  $\beta_M \subset \beta - n(\tau)$.  Right: A corner of $L$ and of
  $R$ meet the tie $\beta_I \subset \beta \cap N(\tau)$.}
\label{Fig:Across}
\end{figure}

\begin{proof}
Suppose, for a contradiction, that $\beta$ is isotopic to $\gamma
\carr \tau$.  We now induct on the intersection number $|\beta \cap
\gamma|$.

%%% Cite bigon criterion?

In the base case $\beta$ and $\gamma$ are disjoint; thus $\beta$ and
$\gamma$ cobound an annulus $A \subset S$.  Since $\beta$ and $\gamma$
are in efficient position with respect to $\tau$, the intersection $A
\cap N(\tau)$ is a union of rectangles, so has index zero.  However,
one of $L$ or $R$ lies inside of $A - N(\tau)$.  This contradicts the
additivity of index.

In the induction step, $\beta$ and $\gamma$ cobound a bigon $B \subset
S$.  Since $\gamma$ is carried, the two corners $x$ and $y$ of $B$ lie
inside of $N(\tau)$.  Let $\beta_x$ be the component of $\beta \cap N$
that contains $x$.  We call $x$ a \emph{carried} or \emph{dual} corner
as $\beta_x$ is a carried arc or a tie.  We use the same terminology
for $y$.

If $x$ is a carried corner then move along $\gamma \cap \bdy B$ a
small amount, let $I_x$ be the resulting tie, and use $I_x$ to cut a
triangle (containing $x$) off of $B$ to obtain $B'$.  Do the same at
$y$ to obtain $B''$.  Now, if both $x$ and $y$ are dual corners then
$B'' = B$ is a bigon.  If exactly one of $x$ or $y$ is a dual corner
then $B''$ is a triangle.  In either of these cases $\ind(B'')$ is
positive, contradicting the assumption that $\beta$ is in efficient
position.

So suppose both $x$ and $y$ are carried corners of $B$; thus $B''$ is
a rectangle.  Thus $B''$ has index zero.  Recall that $\beta_M$ is a
subarc of $\beta$ meeting both $R$ and $L$.  Since neither $L$ or $R$
lie in $B''$ deduce that $\beta_M$ is disjoint from $B''$.  We now
define $\beta_B = \beta \cap B$ and $\gamma_B = \gamma \cap B$, the
two sides of $B$.  We define $\beta'$ to be the curve obtained from
$\beta$ by isotoping $\beta_B$ across $B$, slightly past $\gamma_B$.
So $\beta'$ is isotopic to $\beta$, is in efficient position with
respect to $\tau$, has two fewer points of intersection with $\gamma$,
and contains $\beta_M$.  Thus $\beta_M$ is adjacent to two regions
$L'$ and $R'$ of $S - n(\beta' \cup \tau)$ of negative index, as
desired.  This completes the induction step and thus the proof of the
criterion.
\end{proof}

\refcri{AcrossArc} is not general enough for our purposes.  We also
need a criterion that covers a situation where the regions $L$ and $R$
are not immediately adjacent.  

\begin{criterion}
\label{Cri:AcrossCorner}
Suppose $\beta \eff \tau$ is a curve.  Orient $\beta$.  Suppose there
are regions $L$ and $R$ of $S - n(\beta \cup \tau)$ and a tie $\beta_I
\subset \beta \cap N(\tau)$ with the following properties.
\begin{itemize}
\item
$L$ and $R$ lie to the left and right, respectively, of
  $\beta$, 
\item
the two points of $\bdy \beta_I$ are corners of $L$ and $R$, and 
\item
$L$ and $R$ have negative index. 
\end{itemize}
Then any curve isotopic to $\beta$ is not carried by $\tau$. \qed
\end{criterion}

The proof of \refcri{AcrossCorner} is almost identical to that of
\refcri{AcrossArc} and we omit it.  See \reffig{Across} for local
pictures of curves $\beta \eff \tau$ satisfying the two criteria.
%%% Note that the criteria are essentially local...

\section{Efficient and crossing diagonals}

The next tool needed to prove \refprop{Closing} is the existence of
\emph{crossing diagonals:} efficient arcs that cannot isotoped to be
carried.

%%% The material in this section is generalized by [MMS] and by [T].
%%% But we need so much less than the general case, and this proof is
%%% sufficiently simpler...

Let $\tau$ be a train track.  Suppose $\sigma \subset \tau$ is a
subtrack.  We take $N(\sigma) \subset N(\tau)$ to be a tie
sub-neighborhood, as follows.
%%% Not in italics, because never reused. 
%%% Also - sub-tie-neighborhood looks bad. 
\begin{itemize}
\item
Every tie of $N(\sigma)$ is a subarc of a tie of $N(\tau)$.  
\item
The horizontal boundary $\bdy_h N(\sigma)$ is 
\begin{itemize}
\item
disjoint from $\bdy_h N(\tau)$ and 
\item
transverse to the ties of $N(\tau)$.
\end{itemize}
\item
Every component of $\bdy_v N(\sigma)$ contains a component of $\bdy_v
N(\tau)$.
%%% Thus, just one component, and the containment is proper.
\end{itemize}

Now suppose that $Q$ is a component of $S - n(\sigma)$.  We define
$N(\tau, Q) = N(\tau) \cap Q$.  We say that a properly embedded arc
$\delta \subset Q$ is a \emph{diagonal} of $Q$ if 
\begin{itemize}
\item
$\bdy \delta$ lies in $\bdy_v Q$, missing the corners,
\item
$\delta$ is orthogonal to $\bdy Q$, and
\item
all components of $Q - n(\delta)$ have negative index.
\end{itemize}
A diagonal $\delta$ is \emph{efficient} if it satisfies
\refdef{EffPos} with respect to $N(\tau, Q)$.  An efficient diagonal
$\delta$ is \emph{short} if one component $H$ of $Q - n(\delta)$ is a
hexagon.
% Could define essential and non-periph for diagonals. 
The hexagon $H$ meets three (or two) components of $\bdy_v Q$ and
properly contains one of them, say $v$.  In this situation we say
$\delta$ \emph{cuts} $v$ off of $Q$.  In the simplest example a short
diagonal $\delta \subset Q$ is carried by $N(\tau, Q)$.  

We say an efficient diagonal $\delta$ is a \emph{crossing diagonal} if
there is
\begin{itemize}
\item
a subarc $\delta_M$ (or $\delta_I$) and
\item
regions $L$ and $R$ of $Q - (\delta \cup n(\tau))$
\end{itemize}
satisfying the hypotheses of \refcri{AcrossArc} or
\refcri{AcrossCorner}.  Deduce, if $\beta \eff \tau$ is a curve
containing a crossing diagonal $\delta$, that any curve isotopic to
$\beta$ is not carried by $\tau$.

%%% The usual index argument shows that a crossing
%%% diagonal is essential and non-peripheral in $Q$.  

\begin{lemma}
\label{Lem:ShortDiagonal}
Suppose $\sigma \subset \tau$ is a large subtrack.  Let $Q$ be a
component of $S - n(\sigma)$ that is not a hexagon or a once-holed
bigon.  Then for any component $v \subset \bdy_v Q$ there is a short
diagonal $\delta \subset Q$ that cuts $v$ off of $Q$.

Furthermore $\delta$ is properly isotopic, relative to the corners of
$Q$, to a carried or a crossing diagonal.
\end{lemma}

%%% Warning to typesetters - the following hack was done to place the
%%% wrapfigure correctly.  If you change this you will almost
%%% certainly need to reposition \reffig{Rectangle}.
\noindent\textit{Proof.}
%\begin{proof}
The orientation of $S$ induces an oriention on $Q$ and thus of the
boundary of $Q$.  Let $u$ and $w$ be the components of $\bdy_v Q$
immediately before and after $v$.  (Note that we may have $u = w$.  In
this case $Q$ is a once-holed rectangle.)  Let $h_u$ and $h_w$ be the
components of $\bdy_h Q$ immediately before and after $v$.

Let $N_u$ be the union of the ties of $N(\tau, Q)$ meeting $h_u$.  As
usual, $N_u$ is a union of rectangles.  (See \reffig{NearHorizontal}
for one possibility for $N_u$.)  Let $I$ be a tie of $N_u$, meeting
the interior of $h_u$.  Suppose that $I$ contains a component of
$\bdy_v N_u$.  Thus $I$ locally divides $N_u$ into a pair of
half-rectangles, one large and one small.  When the small
half-rectangle is closer to $v$ than it is to $u$ (along $h_u$), we
say $I$ \emph{faces} $v$.  Among the ties of $N_u$ facing $v$, let
$I_u$ be the one closest to $v$.  (If no tie faces $v$ we take $I_u =
u$.)

\begin{figure}[htbp]
\labellist
\small\hair 2pt
\pinlabel {$u$} [r] at 1 47
\pinlabel {$v$} [l] at 399 47
\pinlabel {$I_u$} [l] at 218 63
\pinlabel {$\delta$} [bl] at 385 83
\endlabellist
\includegraphics[width = 0.9\textwidth]{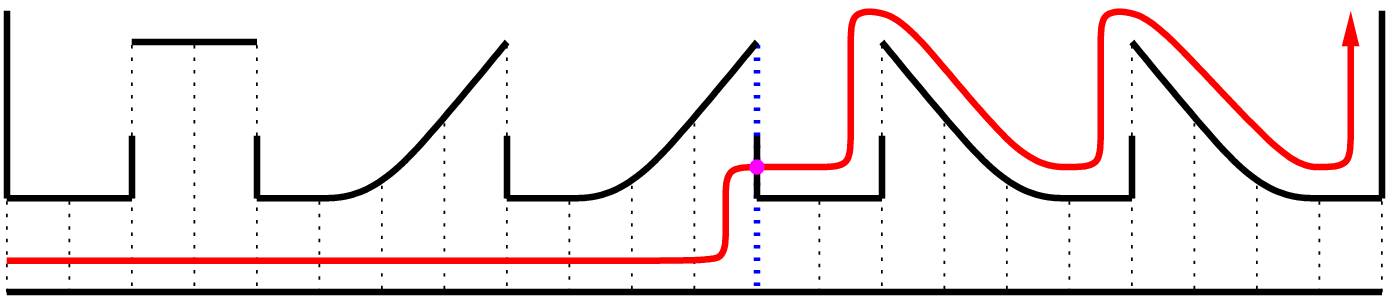}
\caption{One possible shape for $N_u$, the union of all ties meeting
  $h_u \subset \bdy_h Q$.}
\label{Fig:NearHorizontal}
\end{figure}

With $I_u$ in hand, let $N_u'$ be the closure of the component of $N_u
- I_u$ that meets $v$.  We define $I_w$ and $N_w'$ in the same way,
with respect to $h_w$.

\begin{wrapfigure}[13]{r}{0.45\textwidth}
\vspace{-3pt}
\centering 
\labellist
\small\hair 2pt
\pinlabel {$\delta$} [l] at 136 92
\endlabellist
\includegraphics[width = 0.35\textwidth]{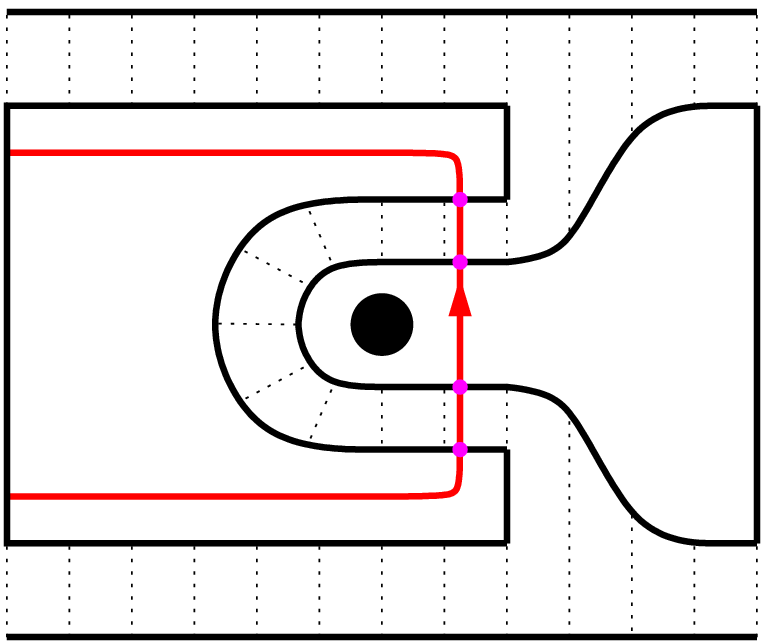}
\caption{We have properly isotoped $\delta$ to simplify the figure.}
\label{Fig:Rectangle}
\end{wrapfigure}

Consider the set $X = h_u \cup N_u' \cup v \cup N_w' \cup h_w$.  Let
$N(X)$ be a small regular neighborhood of $X$, taken in $S$, and set
$\delta = Q \cap \bdy N(X)$; see \reffig{NearHorizontal}.  Note that
$\delta$ cuts $v$ off of $Q$.  Orient $\delta$ so that $v$ is to the
right of $\delta$. 

We now prove that $\delta$ is in efficient position, after an
arbitrarily small isotopy.  The subarc of $\delta_u \subset \delta$
between $u$ and $I_u$ is carried; the same holds for the subarc
$\delta_w$ between $I_w$ and $w$.  (If $I_u = u$ then we take
$\delta_u = \emptyset$ and similarly for $\delta_w$.)  All components
of $\delta \cap N(\tau)$, other than $\delta_u$ and $\delta_w$, are
ties.

Consider $\epsilon = \delta - n(\tau)$.  If $\epsilon$ is connected,
then $\epsilon$ cuts a hexagon $R$ off of $Q - n(\tau)$.  By
additivity of index the region $L \subset Q - (\delta \cup n(\tau))$
adjacent to $R$ has index at most zero.  If $L$ has index zero, it is
a rectangle; we deduce that $\delta$ is isotopic to a carried
diagonal. If $L$ has negative index then $\delta$ is a crossing
diagonal, according to \refcri{AcrossArc}. 

Suppose $\epsilon = \delta - n(\tau)$ is not connected.  We deduce
that the first and last components of $\epsilon$ cut pentagons off of
$Q - n(\tau)$; all other components cut off rectangles.  When $u \neq
w$ then every region of $Q - n(\tau)$ contains at most one component
of $\epsilon$.  In this case an index argument proves that $\delta$ is
a crossing diagonal, according to \refcri{AcrossArc}.  If $u = w$ then
$Q$ is a once-holed rectangle as shown in \reffig{Rectangle}.  In this
case $\delta$ is a crossing diagonal, according to
\refcri{AcrossCorner}. \qed
%\end{proof}
%%% See comments at matching \begin{proof}

\begin{lemma}
\label{Lem:CrossingDiagonal}
Suppose $\sigma \subset \tau$ is a large subtrack.  Let $Q$ be a
component of $S - n(\sigma)$ that is not a hexagon or a once-holed
bigon.  Then $Q$ has a short crossing diagonal.
\end{lemma}

\begin{proof}
Since $\sigma$ is large, $Q$ is a disk or a peripheral annulus.  Set
$n = |\bdy_v Q|$.  According to \reflem{ShortDiagonal}, for every
component $v \subset \bdy_v Q$ there is a short diagonal $\delta_v$
cutting $v$ off of $Q$.  Let $H_v \subset Q - \delta_v$ be the hexagon
to the right of $\delta_v$.  Also every $\delta_v$ is a carried or a
crossing diagonal.

Suppose for a contradiction that $\delta_v$ is carried, for each $v
\subset \bdy_v Q$.  Thus $K_v = H_v - n(\tau)$ is again a hexagon.  If
$u$ is another component of $\bdy_v Q$ then $K_u$ and $K_v$ are
disjoint.  Since index is additive, we find $\ind(Q) \leq
-\frac{n}{2}$.  This inequality is strict when $Q$ is a peripheral
annulus; this is because the component of $Q - n(\tau)$ meeting $\bdy
S$ must also have negative index.

On the other hand, if $Q$ is a disk then $\ind(Q) = 1 - \frac{n}{2}$;
if $Q$ is an annulus then $\ind(Q) = -\frac{n}{2}$.  In either case we
have a contradiction.
\end{proof}

\section{Closing up the diagonal}

After introducing the necessary terminology, we give the proof of
\refprop{Closing}.

Suppose $\tau \subset S$ is a track, and $N = N(\tau)$ is a tie
neighborhood.  Suppose that $I \subset N$ is a tie, containing a
component $u \subset \bdy_v N$.  Let $R$ be the large half-rectangle
adjacent to $I$.  For any unit vector $V(x)$ based at $x \in
\interior(I)$ we say $V(x)$ is \emph{vertical} if it is tangent to
$I$, is \emph{large} if it points into $R$, and is \emph{small}
otherwise.  Suppose $\alpha \carr \tau$ is a carried curve.  A point
$x \in \alpha \cap I$ is \emph{innermost on $I$} if there is a
component $\epsilon \subset I - (u \cup \alpha)$ so that the closure
of $\epsilon$ meets both $u$ and $x$.

Fix an oriented curve $\alpha \carr \tau$.  For any $x \in \alpha$, we
write $V(x, \alpha)$ for the unit tangent vector to $\alpha$ at $x$.
If $x, y \in \alpha$ then we take $[x, y] \subset \alpha$ to be (the
closure of) the component of $\alpha - \{x, y\}$ where $V(x, \alpha)$
points into $[x, y]$.  Note that $\alpha = [x, y] \cup [y, x]$.  Also,
we take $\alpha^{\op}$ to be $\alpha$ equipped with the opposite
orientation.  We make similar definitions when $\alpha$ is a arc.

\begin{proposition}
\label{Prop:Closing}
Suppose $\tau \subset S$ is a train track and $\alpha \carr \tau$ is a
carried curve.  Suppose $\supp(\alpha, \tau)$ is large, but not
maximal.  Then there is a curve $\beta \eff \tau$ so that $i(\alpha,
\beta) \leq 1$ and any curve isotopic to $\beta$ is not carried by
$\tau$.
\end{proposition}

\begin{proof}
Set $\sigma = \supp(\alpha, \tau)$.  Fix a component $Q$ of $S -
n(\sigma)$ that is not a hexagon or a once-holed bigon.  By
\reflem{CrossingDiagonal} there is a short crossing diagonal $\delta
\subset Q$.  Recall that $n(\delta)$ cuts a hexagon $H$ off of $Q$;
also, $\delta$ is oriented so that $H$ is to the right of $\delta$.
The hexagon $H$ meets three components $u, v, w \subset \bdy_v Q$.
The component $v$ is completely contained in $H$; also, we may have $u
= w$.  Let $p$ and $q$ be the initial and terminal points of $\delta$,
respectively.  Thus $p \in u$ and $q \in w$; also $V(p, \delta)$ is
small and $V(q, \delta)$ is large.  (Equivalently, $V(p, \delta)$
points into $Q$ while $V(q, \delta)$ points out of $Q$.)

Let $J_u$ and $J_w$ be the ties of $N(\sigma)$ containing $u$ and $w$.
Rotate $V(p, \delta)$ by $\pi/2$, counterclockwise, to get an
orientation of $J_u$.  We do the same for $J_w$.
%%% This is a bit confusing when u = w as, in this case, we are given
%%% one arc two different orientations.  Anyway.

Since $\sigma = \supp(\alpha, \tau)$, there are pairs of innermost
points $x_R, x_L \in \alpha \cap J_u$ and $z_R, z_L \in \alpha \cap
J_w$.  We choose names so that $x_R, p, x_L$ is the order of the
points along $J_u$ and so that $z_R, q, z_L$ is the order along $J_w$.
Now orient $\alpha$ so that $V(x_R, \alpha)$ is small.

\begin{figure}[htbp]
\labellist
\small\hair 2pt
\pinlabel {$\delta$} [b] at 137 59
\pinlabel {$p$} [r] at 36 57
\pinlabel {$q$} [l] at 238 57
\pinlabel {$x_R$} [br] at 36 31
\pinlabel {$x_L$} [br] at 36 85
\pinlabel {$z_R$} [bl] at 237 31
\pinlabel {$z_L$} [tl] at 237 91
\endlabellist
\includegraphics[width = 0.8\textwidth]{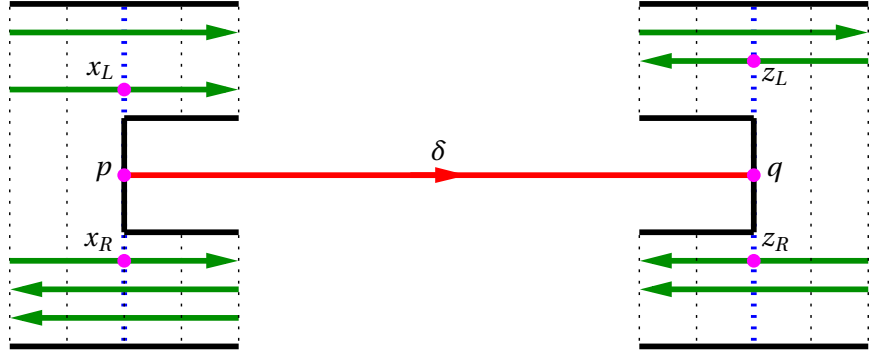}
\caption{In this example, all of the vectors $V(x_R, \alpha)$, $V(x_L,
  \alpha)$, $V(z_R, \alpha)$, and $V(z_L, \alpha)$ are small.}
\label{Fig:Delta}
\end{figure}

We divide the proof into two main cases: one of $V(x_L, \alpha)$,
$V(z_R, \alpha)$, $V(z_L, \alpha)$ is large, or all three vectors are
small.  In all cases and subcases our goal is to construct a curve
$\beta$ which contains $\delta$ and is, after an arbitrarily small
isotopy, in efficient position with respect to $N(\tau)$.  Since
$\beta$ contains $\delta$ one of \refcri{AcrossArc} or
\refcri{AcrossCorner} applies: any curve isotopic to $\beta$ is not
carried by $\tau$.

\subsection{A tangent vector to $\alpha$ at $x_L$, $z_R$, or $z_L$ is large.}

This case breaks into subcases depending on whether or not $u = w$.
Suppose first that $u \neq w$.

If $V(z_R, \alpha)$ is large, then consider the arcs $[q, z_R] \subset
J_w^{\op}$, $[z_R, x_R] \subset \alpha$, and $[x_R, p] \subset J_u$.
The curve
\[
\beta = \delta \cup [q, z_R] \cup [z_R, x_R] \cup [x_R, p]
\]
has the desired properties and satisfies $i(\alpha, \beta) = 0$.

If $V(z_L, \alpha)$ is large, then consider the arcs $[q, z_L] \subset
J_w$, $[z_L, x_R] \subset \alpha$, and $[x_R, p] \subset J_u$.  Then
\[
\beta = \delta \cup [q, z_L] \cup [z_L, x_R] \cup [x_R, p]
\]
has $i(\alpha, \beta) = 1$ because $\beta$ crosses, once, from the
right side to the left side of $[z_L, x_R]$.

Suppose now that $V(z_R, \alpha)$ and $V(z_L, \alpha)$ are small but
$V(x_L, \alpha)$ is large.  Consider the arcs $[p, x_L] \subset J_u$,
$[x_L, z_L] \subset \alpha$, and $[z_L, q] \subset J_w^{\op}$.  Then
\[
\beta = [p, x_L] \cup [x_L, z_L] \cup [z_L, q] \cup \delta^{\op}
\]
has $i(\alpha, \beta) = 0$.

We now turn to the subcase where $u = w$ and $V(x_L,\alpha)$ is large.
In this case $x_R = z_L$, $x_L = z_R$, and the points $x_R, p, q, x_L$
appear, in that order, along $J_u$.  Consider the arcs $[q, x_L]$ and
$[x_R,p] \subset J_u$ and $[x_L,x_R] \subset \alpha$.  Then
\[
\beta = \delta \cup [q, x_L] \cup [x_L, x_R] \cup [x_R, p].
\]
has $i(\alpha, \beta) = 0$. 

\subsection{The tangent vectors to $\alpha$ at $x_L$, $z_R$, and $z_L$
  are small.} 
%%% There is no solution just in $Q$, so we extend $\delta$ further.
Let $R \subset N(\tau)$ be the biggest rectangle, with embedded
interior, where
\begin{itemize}
\item
both components of $\bdy_v R$ are subarcs of ties,
\item
$[z_R, z_L] \subset J_w$ is a component of $\bdy_v R$, and
\item
$\bdy_h R = \alpha \cap R$.
\end{itemize}
Since the interior of $R$ is embedded, the vertical arc $(\bdy_v R) -
[z_R, z_L]$ contains a unique component $u' \subset \bdy_v N(\sigma)$;
also, the component $u'$ is not equal to $w$.  Pick a point $p'$ in
the interior of $u'$.  Let $\epsilon \subset R$ be a carried arc
starting at $q$, ending at $p'$, and oriented away from $q$.

Let $J_{u'}$ be the tie in $N(\tau)$ containing $u'$.  We orient
$J_{u'}$ by rotating $V(p',\epsilon)$ by $\pi/2$, counterclockwise.
Let $x_R'$ and $x_L'$ be the innermost points of $\alpha \cap J_{u'}$.
Note that $V(x_L', \alpha)$ and $V(x_R', \alpha)$ are both large.
Thus $u' \neq u$.  We have already seen that $u' \neq w$.

Let $Q'$ be the component of $S - n(\delta \cup \sigma)$ that contains
$u'$.  The orientation on $S$ restricts to $Q'$, which in turn induces
an orientation on $\bdy Q'$.  Let $v'$ be the component of $\bdy_v Q'$
immediately \emph{before} $u'$.

If $v' = u'$ then $Q'$ is a once-holed bigon, contradicting the fact
that $V(x_L', \alpha)$ and $V(x_R', \alpha)$ are both large.  

If $v' \subset u - \delta$ then $Q' = H \subset Q$ is the hexagon to
the right of $\delta$.  Thus $u' = v$.  In this case there is a curve
$\alpha' \subset R \cup H$ so that 
\begin{itemize}
\item
$\alpha' \cap R$ is a properly embedded arc with endpoints $z_R$ and
  $x'_R$ and
\item
$\alpha' - R$ is a component of $\bdy_h H$. 
\end{itemize}
Thus $\alpha'$ is isotopic to (the right side of) $\alpha$.  Now, if
$u \neq w$ then $\alpha'$ also meets the region $Q - (n(\delta) \cup
H)$, near $x_L$.  Thus $\alpha'$ is not contained in $R \cup H$, a
contradiction.  If $u = w$ then $Q$ is a once-holed rectangle.  In
this case $R \cup Q$, together with a pair of rectangles, is all of
$S$.  Thus $S$ is a once-holed torus, contradicting our standing
assumption that $\chi(S) \leq -2$.

%%% Note that we have already disposed of the case where S is a
%%% four-holed sphere -- in a four-holed sphere if \sigma is large
%%% then it is maximal.

If $v' \subset w - \delta$ then $Q = Q' \cup N(\delta) \cup H$.  We deduce
that $Q$ is not a once-holed rectangle; so $u \neq w$.  Also, the left
side of $\alpha$ is contained in $R \cup Q'$.  However the left side
of $\alpha$ meets the hexagon $H$, near the point $x_R$, giving a
contradiction.

To recap: the arc $\delta \cup \epsilon$ enters $Q'$ at $p' \in u'
\subset \bdy_v Q'$.  The region $Q'$ is not a once-holed bigon; also
$Q' \cap H = \emptyset$.  The component $v' \subset \bdy_v Q'$ coming
before $u'$ is not contained in $w - \delta$.

Let $J_{v'}$ be the tie of $N(\tau)$ containing $v'$.  We orient
$J_{v'}$ using the orientation of $Q'$.  Let $y_R'$ and $y_L'$ be the
two innermost points of $\alpha \cap J_{v'}$, where $y_R'$ comes
before $y_L'$ along $J_{v'}$.  Since $V(x_L', \alpha)$ is large, the
vector $V(y_L', \alpha)$ is small.  

We now have a final pair of subcases.  Either $V(y_R', \alpha)$ is
large, or it is small.

\subsection{The tangent vector to $\alpha$ at $y_R'$ is large.}

Consider the arcs $[p', x_L'] \subset J_{u'}$, $[x_L', y_L'] \subset
\alpha^\op$, $[y_L', y_R'] \subset J_{v'}^\op$, $[y_R', x_R] \subset
\alpha$, and $[x_R, p] \subset J_u$.  Then
\[
\beta = \delta \cup \epsilon \cup [p', x_L'] \cup [x_L', y_L'] \cup
      [y_L', y_R'] \cup [y_R', x_R] \cup [x_R, p]
\]
has $i(\alpha, \beta) = 0$.  (Note that after an arbitrarily small
isotopy the arc $[p', x_L'] \cup [x_L', y_L'] \cup [y_L', y_R']$
becomes carried.)

\subsection{The tangent vector to $\alpha$ at $y_R'$ is small.}
\label{Sec:Final}

In this case we consider the component $w'$ of $\bdy_v Q'$ immediately
before $v'$.  Recall that $Q' \cap H = \emptyset$.  Now, if $w'$ is
the left component of $w - \delta$ then $V(y_R', \alpha)$ being small
implies $V(z_L, \alpha)$ is large, contrary to assumption.  If $w'$ is
the left component of $u - \delta$ then $v'$ is contained in $w$, a
contradicton.

As usual, let $J_{w'}$ be the tie in $N(\tau)$ containing $w'$.  Since
$w'$ is not contained in $u$ or $w$ there are a pair of innermost
points $z_R'$ and $z_L'$ along $J_{w'}$.  Applying
\reflem{ShortDiagonal} there is a short diagonal $\delta'$ in $Q'$
that
\begin{itemize}
\item
connects $p' \in u'$ to a point $q' \in w'$ and
\item
cuts $v'$ off of $Q'$.
\end{itemize}
Note that $v'$ is to the \emph{left} of $\delta'$.    

We give $J_{w'}$ the orientation coming from $\bdy_v Q'$; this agrees
with the orientation given by rotating $V(q', \delta')$ by angle
$\pi/2$, counterclockwise.  We choose names so the points $z_R', q',
z_L'$ come in that order along $J_{w'}$.

Consider the arcs $[q', z_L'] \subset J_{w'}$, $[z_L', x_R] \subset
\alpha$, and $[x_R, p] \subset J_{u}$.  Then
\[
\beta = \delta \cup \epsilon \cup \delta' \cup [q', z_L'] \cup
      [z_L', x_R] \cup [x_R, p]
\]
has $i(\alpha, \beta) = 1$ because $\beta$ crosses, once, from the
right side to the left side of $[z_L', x_R]$.  This is the final case,
and completes the proof.
\end{proof}

\begin{figure}[htbp]
\labellist
\small\hair 2pt
\pinlabel {$Q$} [tl] at 4 92
\pinlabel {$Q'$} [br] at 360 24
\pinlabel {$\delta$} [r] at 31 59
\pinlabel {$p$} [l] at 111 76
\pinlabel {$q$} [Bl] at 111 44
\pinlabel {$\delta'$} [l] at 333 59
\pinlabel {$p'$} [Br] at 254 80
\pinlabel {$q'$} [r] at 254 40
\pinlabel {$\epsilon$} [br] at 179 58
\pinlabel {$x_R = z_L$} [b] at 110 115
\pinlabel {$x_L = z_R$} [t] at 110 -1
\pinlabel {$x'_R = z'_L$} [t] at 254 2
\pinlabel {$x'_L = z'_R$} [b] at 254 114
\endlabellist
\vspace{0.15cm}
\includegraphics[width = 0.8\textwidth]{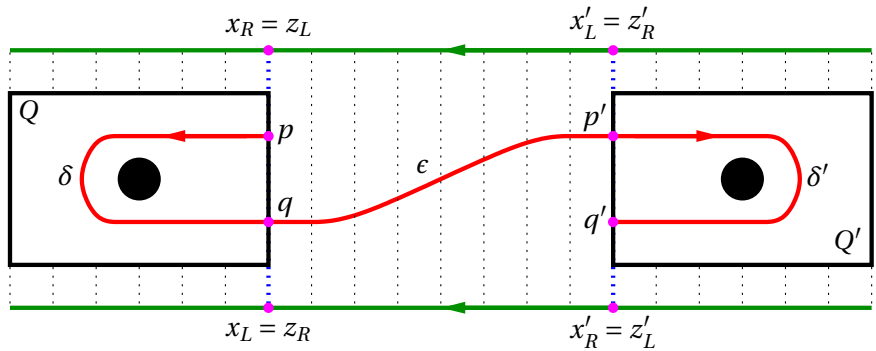}
\vspace{0.1cm}
\caption{One of the four possibilities covered by \refsec{Final}.
  Here $u = w$ and $u' = w'$, so both of $Q$ and $Q'$ are once-holed
  rectangles.}
\label{Fig:DeltaEps}
\end{figure}

%%% In the case shown (u = w and u' = w') there is a much simpler
%%% solution using (\delta')^\op and messing about with endpoints.
%%% But the above construction covers all four cases. 

\bibliographystyle{hyperplain} % My substitute for plain.bst
\bibliography{bibfile}
\end{document}